\documentclass[12pt,reqno]{amsart}
\usepackage{color,soul}
\usepackage{amsfonts} \usepackage{amssymb} \usepackage{verbatim}

\usepackage{thmtools}
\usepackage{thm-restate}
\usepackage{hyperref}
\usepackage{cleveref}
 \title  [Interpolation on Certain Subalgebras of
$H^{\infty}(\D)$] {Nevanlinna-Pick Interpolation on Certain Subalgebras of $H^{\infty}(\D)$}
\author{Debendra P. Banjade and Jeremiah Dunivin}

\address{Department of Mathematics and Statistics\\
         Coastal Carolina University\\
         Box 261954\\
         Conway, SC 29528\\
         843-3496569}
\email{dpbanjade@coastal.edu, jdunivi1@coastal.edu}

\subjclass[2010]{Primary: 30H05; Secondary:32A38}
\keywords{Nevanlinna - Pick interpolation theorem, bounded analytic functions, sub algebras}

\newtheorem{remark}{Remark}
\newtheorem{thm}{Theorem}[section]

\newtheorem*{Bruno}{Bruno's Formula}
\newtheorem{prop}{Proposition}[section]
\newtheorem{cor}{Corollary}[section]

\newtheorem{Lemma}{Lemma}[section]

\newtheorem*{pick}{Nevanlinna-Pick Interpolation Theorem} 
\newtheorem*{david}{Interpolation Theorem on $H^{\infty}_{\{1\}}(\mathbb {D})$ (Davidson, Paulsen, Raghupathi, and Singh)}

    \newcommand{\C}{\mathbb{C}}
  
 \newcommand{\D}{\mathbb{D}} \newcommand{\Z}{\mathbb{Z}}

\newtheorem{definition}{Definition}
\newtheorem{exmp}{Example}[section]
\newtheorem{question}{Question}[section]
\allowdisplaybreaks[1]
\begin{document}

\maketitle

\begin{abstract} 
 Given a collection $K$ of positive integers, let $H^{\infty}_K(\mathbb{D})$ denote the set of all bounded analytic functions defined on the unit disk $\mathbb{D}$ in $\C$ whose $k^{\text{th}}$ derivative vanishes at zero, for all $k \in K$. In this paper, we establish a Nevanlinna-Pick interpolation result for the subalgebra $H^{\infty}_K(\mathbb{D})$, where $K = \{1,2,\dotsc,k\}$, which is a slight generalization of the interpolation theorem that Davidson, Paulsen, Raghupathi, and Singh proved for the algebra $H^{\infty}_{\{1\}}(\mathbb{D})$. Furthermore, we provide a sufficient condition for an interpolation function to exist in the algebra $H^{\infty}_K(\mathbb{D})$ for a given $K$. Lastly, we give a necessary condition for the existence of such interpolation functions. 
\end{abstract}

\section{Introduction} 
The algebras we are interested  in this paper are defined as follows: \\
Let $ K \subset \Z_{+} $ and define $$ H_K^{\infty}(\D) = \lbrace f
\in H^{\infty}(\D) : f^{(k)}(0) = 0,\;  \text{ for \; all}\;  \;  k \in K \rbrace. $$ We consider those sets $ K $ for which $
H_K^{\infty}(\D) $ is an algebra under pointwise multiplication of functions. Obviously, not every set $ K
$ defines an algebra (for example, let $ K = \lbrace 2 \rbrace $). Though there is not a complete characterization of  the set $K$ for which $H_K^{\infty}(\D)$ is an algebra, Ryle and Trent \cite{ryledissertation} have given certain criteria that the set $K$ must meet. We state some of these criteria as Lemmas \ref{rylelemma} and \ref{Kinfinite} in Section 3.

Furthermore, we endow $H^{\infty}(\mathbb{D})$ with the usual supremum norm, which we define below.
\begin{definition}
\normalfont Let $f$ be a function defined on $\mathbb{D}$. The \textit{supremum norm} $||f||_{\infty}$ of $f$ is defined as the least upper bound for the set of all moduli of $f$. Formally, we write
$$||f||_{\infty} = \sup\{|f(z)| : z \in \mathbb{D}\}.$$
\end{definition}

Nevanlinna and Pick proved an interpolation result for the algebra $H^{\infty}(\mathbb{D})$ in \cite{Nevanlinna} and \cite{nikolski}. Their theorem shows that an interpolant exists in this algebra precisely when the associated Pick matrix is positive semidefinite. 

\begin {pick} Given distinct points $z_1, z_2, ..., z_n$ in the unit disk $\mathbb{D}$ in the complex plane and complex numbers $w_1, w_2, ..., w_n $ in $\mathbb{D},$ there exists a function $f$ in $H^{\infty}(\mathbb {D})$ with $f(z_i) = w_i$ for $i=1,2,\dotsc,n$ and $\Vert f \Vert _{\infty}  \leq 1$ if and only if the associated Pick matrix 
$$ \left [  \frac{1- w_i \overline{w_j}}{ 1 - z_i \overline{z_j} } \right] _ {i , j} \; \text{ is \; positive \; semidefinite}. $$
\end{pick}
 Davidson, Paulsen, Raghupathi, and Singh \cite{david} proved an interpolation result for  $H^{\infty}_{\{1\}}(\mathbb{D})$, the collection of all bounded, analytic functions $f$ on $\mathbb{D}$, satisfying the constraint $f^{\prime}(0) = 0$ (see Theorem 1.2 in \cite{david}).  Similar research on $H^{\infty}_{\{1\}}(\mathbb{D})$ can be found in \cite{Michael1} and \cite{Michael}.

\begin {david} \label{rags} Let  $z_1, z_2, ..., z_n$ be distinct points in  $\mathbb{D},$ and let  $w_1, w_2, ..., w_n,$ be in  $\mathbb{D}$. Then there  exists an analytic function  $f$ in $H^{\infty}_{\{1\}}(\D)$ with $ \Vert f \Vert _{\infty}  \leq 1$ such that $f(z_i) = w_i $ for $i = 1,2,\dotsc,n$  if and only if there exists $\lambda \in \mathbb{D} $ so that 
$$ \left [  \frac{z_i^2 \overline{z_j ^2}- \varphi _\lambda (w_i) \overline{ \varphi _\lambda (w_j)}  } { 1 - z_i \overline{z_j} } \right] _ {i , j} \; \text{ is \; positive \; semidefinite}, $$
where $\varphi_{\lambda}(z) = \frac{z-\lambda}{1-\overline{\lambda}z}$ denotes the elementary M\"{o}bius transformation.
\end{david}
 Observing that all functions in $H^{\infty}_{\{1\}}(\mathbb{D})$ have first derivatives that vanish at zero, it is natural to define our desired set as follows.
\begin{definition}
Let $K$ denote a nonempty collection of positive integers and define the set $H^{\infty}_K(\mathbb{D})$ as the set of all functions in $H^{\infty}(\mathbb{D})$ whose $k^{\text{th}}$ derivative vanishes at zero, for all $k \in K$. Formally, define
$$ H_K^{\infty}(\D) := \lbrace f
\in H^{\infty}(\D) : f^{(k)}(0) = 0,\;  \text{ for all}\;  \;  k \in K \rbrace. $$
\end{definition}
\noindent We assume that  $H^{\infty}_K(\mathbb{D})$ is an algebra. As we shall see in Section 2, this assumption is necessary.

  Although \cite{david} is the primary inspiration for this work, Carleson's corona theorem and Wolff's problem of ideals have also been established for $H^{\infty}_K(\mathbb{D})$, see \cite{ryle} and \cite{banjade}, respectively. Thus, we found it plausible that a similar interpolation condition holds for the subalgebra $H^{\infty}_K(\mathbb{D})$.   Indeed, we will prove several interpolation results for $H^{\infty}_K(\mathbb{D})$.
  
  In the next section, we will briefly consider examples of sets $K$ that may or may not yield an algebra, demonstrating that we must make the assumption that $H^{\infty}_K(\mathbb{D})$ is an algebra. Section 3 is devoted to all results that will be used to assist us in proving Theorems \ref{newtheorem1}, \ref{newtheorem2}, and \ref{newtheorem3} and Corollaries \ref{newcorollary1} and \ref{newcorollary2} in Section 4 of this paper.
  
\section{The Subalgebra $H^{\infty}_K(\mathbb{D})$}\label{sect2}
We consider those sets $K$ for which $H^{\infty}_K(\mathbb{D})$ is an algebra. As the next example shows, not all sets $K$ yield an algebra.
\begin{exmp}
\normalfont The set $H^{\infty}_{\{2\}}(\mathbb{D})$ is not an algebra, because it fails to be closed under pointwise multiplication. Consider the functions $f(z) = z$ and $g(z)  = z$ on $\mathbb{D}$. Clearly, $f$ and $g$ are analytic and bounded, so $f,g \in H^{\infty}(\mathbb{D})$. Further, $f^{\prime\prime}(0) = g^{\prime\prime}(0) = 0$, so $f,g \in H^{\infty}_{\{2\}}(\mathbb{D})$. But for the product $f(z)g(z) = z^2$, we have $\big(fg)^{\prime\prime}(z) = 2$, which does not vanish at zero. Thus, $fg \notin H^{\infty}_{\{2\}}(\mathbb{D})$, so it is not closed under pointwise multiplication. Therefore,  $H^{\infty}_{\{2\}}(\mathbb{D})$ is not an algebra.
\end{exmp}
\noindent However, the set $H^{\infty}_{\{1,3\}}(\mathbb{D})$ is an algebra.
\begin{exmp}
\normalfont Let $f, g \in H^{\infty}_{\{1,3\}}(\mathbb{D})$. For the sum $f+g$, we have
$$(f+g)^{\prime}(0) = f^{\prime}(0) + g^{\prime}(0) = 0 \hspace{3mm} \text{ and } \hspace{3mm} (f+g)^{\prime\prime\prime}(0) = f^{\prime\prime\prime}(0) + g^{\prime\prime\prime}(0) = 0,$$
so $f+g \in H^{\infty}_{\{1,3\}}(\mathbb{D})$. For the first derivative of the product $fg$,
$$(fg)^{\prime}(z) = f^{\prime}(z) g(z) + f(z) g^{\prime}(z),$$
which implies that
\begin{flalign*}
(fg)^{\prime}(0) &= f^{\prime}(0)g(0) + f(0) g^{\prime}(0)\\
&= 0 \cdot g(0) + f(0) \cdot 0\\
&= 0.
\end{flalign*}
Similarly, the third derivative of $fg$ yields
$$(fg)^{\prime\prime\prime}(z) = f^{\prime\prime\prime}(z)g(z) + 3f^{\prime\prime}(z)g^{\prime}(z) + 3f^{\prime}(z)g^{\prime\prime}(z) + f(z)g^{\prime\prime\prime}(z),$$
and thus
\begin{flalign*}
(fg)^{\prime\prime\prime}(0) &= f^{\prime\prime\prime}(0)g(0) + 3f^{\prime\prime}(0)g^{\prime}(0) + 3f^{\prime}(0)g^{\prime\prime}(0) + f(0)g^{\prime\prime\prime}(0)\\
&= 0 \cdot g(0) + 3f^{\prime\prime}(0) \cdot 0 + 3 \cdot 0 \cdot g^{\prime\prime}(0) + f(0) \cdot 0 \\
&= 0.
\end{flalign*}
Therefore, $fg \in H^{\infty}_{\{1,3\}}(\mathbb{D})$. Since $H^{\infty}_{\{1,3\}}(\mathbb{D})$ is closed under both poinwise multiplication and addition, we conclude that it is an algebra.
\end{exmp}

The above examples show that not every $K$ yields an algebra. A complete characterization of such sets $K$ is an open problem worthy of investigation. Nevertheless, we will proceed to assume that a $K$ is given such that $H^{\infty}_K(\mathbb{D})$ is an algebra. 
\section{Preliminaries}\label{sect3}

Ryle and Trent have conducted a thorough analysis on the subalgebra $H^{\infty}_K(\mathbb{D})$, revealing the rich structure and interesting properties it possesses (see Chapter 4 in \cite{ryledissertation}). Furthermore, they provide necessary criteria that $K$ must satisfy whenever $H^{\infty}_K(\mathbb{D})$ is an algebra. We catalog some of these attributes below (see Lemmas 4.2.1, 4.2.2, 4.2.4, and Corollary 4.2.1 in \cite{ryledissertation}).

\begin{Lemma}{(\textbf{Ryle and Trent})}\label{rylelemma}
Let $K\subset \mathbb{Z}_+$ such that $ H_K^{\infty}(\D)$ is an algebra. Then 
\begin{align*}
& (i) \;\; {k_0}\not\in K \;\; \text{if and only if } \;\; \varphi (z)=z^{k_0}\in  H_K^{\infty}(\D).\\
& (ii)\text{ If} \;\; j ,k \notin K, \text{then}\;\;  j+k \notin K.\\
& (iii) \text{ If} \;\; j \notin K, \text{then}\;\; nj \notin K \;\; \text{ for all} \;\; n \geq 2.\\
& (iv) \text{ Suppose} \;\; {k_0}\in K. \; \text{If}\;\;  1<j< k_0 \;\; \text{satisfies} \;\; j\notin K,\;\;  \text{then}\;\; k_0-j\in K.
\end{align*}
\end{Lemma}
 Given an algebra $H^{\infty}_K(\mathbb{D})$, the following lemma demonstrates that some collection of positive integers absent from $K$ exists. (see Corollay $1$ in \cite{ryle2}).
\begin{Lemma}(\textbf{Ryle and Trent)}\label{Kinfinite} If $ H_K^{\infty}(\D) $ is an algebra, then there exists $ d \in \mathbb{Z}_+ $, a finite
set $ \lbrace n_i \rbrace_{i=1}^p \subset \mathbb{Z}_+ $ with $ n_1 < \dots < n_p $ and $ gcd(n_1,
\dots ,n_p) = 1 $, and a positive integer $ N_0 > n_p $ so that \begin{equation*} \mathbb{Z}_+ \setminus K =
\lbrace n_1d, n_2d, \dots , n_pd, N_0d, (N_0 + j)d : j \in \mathbb{Z}_+ \rbrace \text{.}
\end{equation*} 
\end{Lemma}

These lemmas show that every function $f \in H_K^{\infty}(\D) $ has the following power series representation: 
\begin{flalign*} 
f(z) &= \alpha_0 + \alpha_{1}\,z^{n_1d} + \alpha_{2}\,z^{n_2d} + 
\dots + \alpha_{p}\,z^{n_pd} + \alpha_{p+1}\,z^{N_0d} +  \sum_{j=1}^{\infty}\alpha_{p+j+1}\,z^{(N_0+j)d},
\end{flalign*} 
where the coefficient $ \alpha_i
\in \C $. In order to improve our understanding of these results, we provide a couple of  examples below.

\begin{exmp}\normalfont
\text{\hspace{3mm}}

\vspace{1mm}

\begin{itemize}
    \item[i.] If $f \in H^{\infty}_{\{1\}}(\mathbb{D})$, then the power series representation of $f$ is 
    $$f(z) = \alpha_0 + \alpha_2 z^2 + \alpha_3z^3 + \alpha_4z^4 + \sum_{i=5}^{\infty}\alpha_iz^i.$$
    \item[ii.] If $f \in H^{\infty}_{\{1,3\}}(\mathbb{D})$, then the power series representation of $f$ is  
    $$f(z) = \alpha_0 + \alpha_2z^2 + \alpha_4z^4 + \alpha_5z^5 +\sum_{i=6}^{\infty}\alpha_iz^i.$$
\end{itemize}
\end{exmp}
Because we will be taking derivatives of function compositions, we will be using Bruno's Formula below (see \cite{johnson} for more information).
\begin{Bruno}
Let $f$ and $g$ denote functions with a sufficient number of derivatives. Then the $k^{\text{th}}$ derivative of the composition $g\circ f$ is 
$$\frac{d^k}{dz^k}\Big[ g(f(z)) \Big] = \sum \frac{k!}{b_1! \, b_2! \, \cdots \, b_k!} \cdot g^{(b)}(f(z)) \cdot \prod_{\ell=1}^k \Big(\frac{f^{(\ell)}(z)}{\ell!}\Big)^{b_{\ell}}$$
where the sum is over all $k$-tuples of nonnegative integers $(b_1,\dotsc,b_k)$ satisfying the constraint
$$\sum_{\ell=1}^k \ell\,b_{\ell} = k,
$$
and $b:=b_1+b_2+\cdots+b_k.$
\end{Bruno}


 We use Bruno's formula in the following example, which serves as a motivation for Lemmas \ref{jay1} and \ref{jay2} below.

\begin{exmp}\label{example}
\normalfont When $k=3$, Bruno's formula shows that we take the triples $(b_1,b_2,b_3)$ satisfying the equation $b_1 + 2b_2 + 3b_3 = 3.$ Clearly, the only solutions are $(3,0,0)$, $(1,1,0)$ and $(0,0,1)$. Given $(3,0,0)$, the first term has the form
$$\frac{3!}{3! 0! 0!} \cdot g^{(3+0+0)}(f(z)) \cdot \bigg(\frac{f^{(1)}(z)}{1!}\bigg)^{3} \bigg(\frac{f^{(2)}(z)}{2!}\bigg)^0 \bigg(\frac{f^{(3)}(z)}{3!}\bigg)^0 = g^{(3)}(f(z))\big[f^{(1)}(z)\big]^3.$$
Further, the solution $(1,1,0)$ yields the term
$$\frac{3!}{1! 1! 0!} \cdot g^{(1 + 1 + 0)}(f(z)) \cdot \bigg(\frac{f^{(1)}(z)}{1!}\bigg)^{1} \bigg(\frac{f^{(2)}(z)}{2!}\bigg)^1 \bigg(\frac{f^{(3)}(z)}{3!}\bigg)^0 = 3g^{(2)}(f(z))f^{(1)}(z)f^{(2)}(z).$$
Finally, the solution $(0,0,1)$ provides the last term of the form
$$\frac{3!}{0! 0! 1!} \cdot g^{(0 + 0 + 1)}(f(z)) \cdot \bigg(\frac{f^{(1)}(z)}{1!}\bigg)^{0} \bigg(\frac{f^{(2)}(z)}{2!}\bigg)^0 \bigg(\frac{f^{(3)}(z)}{3!}\bigg)^1 = g^{(1)}(f(z))f^{(3)}(z).$$
Therefore,
$$(g\circ f)^{(3)}(z) = g^{(3)}(f(z))\big[f^{(1)}(z)\big]^3 + 3g^{(2)}(f(z))f^{(1)}(z)f^{(2)}(z) + g^{(1)}(f(z))f^{(3)}(z).$$

Note that each solution to the constraint generates a term in the above equation. Now, consider the algebra $H^{\infty}_{\{1,3\}}(\mathbb{D})$. If $f \in H^{\infty}_{\{1,3\}}(\mathbb{D}),$ we see from the first three equations that a nonzero component $b_j$ in $(b_1,b_2,b_3)$ exists such that $j \in \{1,3\}$. We show that this is always true.
\end{exmp}
\begin{Lemma}\label{jay1}
Suppose $K$ is a set of positive integers such that $H^{\infty}_K(\mathbb{D})$ is an algebra. If $k \in K$, then for every $k$-tuple $(b_1,b_2,\dotsc,b_k)$ satisfying
$$\sum_{\ell=1}^k \ell \, b_{\ell} = k,$$
there exists an $\ell_j \in \{\ell_1, \ell_2, \dotsc, \ell_t\}$ such that $\ell_j \in K$, where $b_{\ell_1}, b_{\ell_2}, \dotsc, b_{\ell_t}$ are the nonzero components of $(b_1, b_2, \dotsc, b_k)$.
\end{Lemma}
\begin{proof}
Let $k \in K$. Assume instead that $\ell_j \notin K,$ for all $\ell_j \in \{\ell_1,\ell_2,\dotsc,\ell_t\}$. By Lemma \ref{rylelemma}.(iii), the product $\ell_j \, b_{\ell_j} \notin K$. Further, 
$$\sum_{j=1}^t \ell_j \, b_{\ell_j} \notin K,$$
by Lemma \ref{rylelemma}.(ii). Since $b_m = 0$ for each $m \in \{1,2,\dotsc, k\}\setminus\{\ell_1,\ell_2,\dotsc,\ell_t\}$,

\begin{flalign*}
\sum_{\ell = 1}^k \ell \, b_{\ell} = \sum_{j=1}^t \ell_j \, b_{\ell_j} = k \notin K,
\end{flalign*}
which is a contradiction, as we know $k \in K$.\\
\end{proof}

\begin{Lemma}\label{jay2}
Suppose that  $H^{\infty}_K(\mathbb{D})$ is an algebra.  If $g \in H^{\infty}(\mathbb{D})$ and $f \in H^{\infty}_K(\mathbb{D})$ with $f(\mathbb{D})\subset\mathbb{D}$, then $g \circ f \in H^{\infty}_K(\mathbb{D})$.
\end{Lemma}
\begin{proof}
Let $g \in H^{\infty}(\mathbb{D})$ and $f \in H^{\infty}_K(\mathbb{D})$. Then $g \circ f \in H^{\infty}(\mathbb{D})$. If $k \in K$, Bruno's formula allows us to write 
$$\big(g(f(z)))^{(k)} = \sum \frac{k!}{b_1! \, b_2! \, \cdots \, b_k!} \cdot g^{(b)}(f(z)) \cdot \prod_{\ell=1}^k \Big(\frac{f^{(\ell)}(z)}{\ell!}\Big)^{b_{\ell}}$$
where the sum is over all $k$-tuples of nonnegative integers $(b_1,\dotsc,b_k)$ satisfying the constraint
$$\sum_{\ell=1}^k \ell\,b_{\ell} = k,
$$
and $b:=b_1+b_2+\cdots+b_k.$ For each solution $(b_1,b_2,\dotsc,b_k)$, there exist nonzero components $b_{\ell_1}, b_{\ell_2}, \dotsc, b_{\ell_t}$ in $(b_1,b_2,\dotsc,b_k)$, because $k>0$. By Lemma \ref{jay1}, there exists an $\ell_j \in \{\ell_1, \ell_2, \dotsc, \ell_t\}$ such that $\ell_j \in K$, and thus
$$\bigg(\frac{f^{(\ell_j)}(z)}{\ell_j!}\bigg)^{b_{\ell_j}} = 0 \; \text{ at }\; z = 0.$$
Therefore, every term in the sum vanishes at $z=0$, so $(g(f(0)))^{(k)} = 0.$ Hence, $g \circ f \in H^{\infty}_K(\mathbb{D})$.
\end{proof}

The following propositions will be used to prove Lemma \ref{jay4} below. We first show that if $H^{\infty}_K(\mathbb{D})$ is an algebra, $1 \in K$.
\begin{prop}\label{newlemma1}
If $K\subset \mathbb{Z}_+$ is nonempty such that $H^{\infty}_K(\mathbb{D})$ is an algebra, then $1 \in K$.
\end{prop}
\begin{proof}
Assume to the contrary that $1 \notin K$. Let $k$ be the smallest positive integer in $K$. Then $k-1 \notin K$ where $1 \leq k-1$. If $k-1 = 1$, then $k = 2 \in K$. By Lemma \ref{rylelemma}.(i), $z \in H^{\infty}_K(\mathbb{D})$ and $z^2 \notin H^{\infty}_K(\D)$. However, $H^{\infty}_K(\mathbb{D})$ is an algebra, so $z\cdot z = z^2 \in H^{\infty}_K(\mathbb{D})$, a contradiction. Therefore, $1 < k-1$. But since $k-1 < k$ with $k-1 \notin K$ and $k \in K$, Lemma \ref{rylelemma}.(iv) implies that $k-(k-1) = 1 \in K$, another contradiction. Thus, we must have $1 \in K$. 
\end{proof}
If $K$ only has consecutive integers, then either $K$ is finite, containing all positive integers less than some integer, or $K = \mathbb{Z}_+$.

\begin{prop}\label{newlemma2}
Suppose $H^{\infty}_K(\mathbb{D})$ is an algebra. If $K$ only contains consecutive integers, then $K = \{1,2,\dotsc,k\}$ for some integer $k$, or $K = \mathbb{Z}_+$. 
\end{prop}
\begin{proof}
Suppose for each $k \in \mathbb{Z}_+$, $K \neq \{1,2,\dotsc,k\}$. By Proposition \ref{newlemma1}, $1 \in K$. But $K\neq\{1\}$. Since $K$ only has consecutive integers, $2 \in K$. Similarly, if $k \in K$, then $k+1 \in K$. Thus, $K = \mathbb{Z}_+$. 
\end{proof}
Lastly, we show that if there is a smallest integer $m$ not in $K$, then $K$ contains all positive integers less than $m$.
\begin{prop}\label{newlemma3}
Let $H^{\infty}_K(\mathbb{D})$ be an algebra with $K\neq\mathbb{Z}_+$. Let $m$ be the smallest positive integer not in $K$. Then $K$ contains every positive integer $j < m$.
\end{prop}
\begin{proof}
By Proposition \ref{newlemma1}, $K\neq\emptyset$. Since $K$ and $\mathbb{Z}_+\setminus K$ partition $\mathbb{Z}_+$ and $m = \min \big(\mathbb{Z}_+\setminus K\big)$, every positive integer $j < m$ must be in $K$. 
\end{proof}
We show that a smallest positive integer $m \notin K$ can be found, so that a function $f$ can be written as a product of functions involving $m$.

\begin{Lemma}\label{jay4}
Let $K \subset \mathbb{Z}_+$ be a proper nonempty set such that $H^{\infty}_K(\mathbb{D})$ is an algebra. Let $f \in H^{\infty}_K(\mathbb{D})$ with $f(0)=0$. Then there exists a smallest positive integer $m \notin K$  and an analytic function $h$ on $\D$ such that
$$f(z) = z^m h(z).$$
\end{Lemma}
\begin{proof}
It is enough to assume that  $f$ is not identically zero and nonconstant on $\mathbb{D}$. If $K$ only contains consecutive integers, then because $K$ is a proper subset of $\mathbb{Z}_+$, Proposition \ref{newlemma2} tells us that $K = \{1,2,3,4,\dotsc,k\}$, for some $k \in \mathbb{Z}_+$. Then the smallest positive integer not in $K$ is $k + 1$. Now, if $f \in H^{\infty}_K(\mathbb{D})$ and  $f(0) = 0$, it follows that $f$ has the following representation:
\begin{flalign*}
f(z) &= \alpha_{k + 1}z^{k + 1} + \sum_{i=k + 2}^{\infty} \alpha_iz^i\\
&= z^{k+1}h(z),
\end{flalign*}
where $h(z) = \alpha_{k + 1} + \sum_{k + 2}^{\infty}\alpha_i z^i$ is analytic, as desired.

If $K$ contains a pair of nonconsecutive integers, let $k_i$ and $k_{i+1}$ denote the smallest such pair. Then $m = k_i + 1$ is the smallest positive integer not in $K$. By Proposition \ref{newlemma3} every positive integer $j < k_i+1$ is in $K$, so that $f^{(j)}(0) = 0$. Because $f(0) = 0$, we know that $f$ has the following representation for all positive integers $\ell \notin K \cup \{m\}$:
\begin{flalign*}
f(z) &= \alpha_m z^m + \sum_{\ell} c_{\ell} z^{\ell} \\[3pt]
     &= z^m h(z),
\end{flalign*}
where $h(z) = \alpha_m + \sum\nolimits_{\ell} \, c_{\ell}\, z^{\ell-m}$ is an analytic function, as desired.
\end{proof}
Let $f$ be an analytic function on $\mathbb{D}$ such that $||f||_{\infty}\leq 1$. If there is a point $z_0 \in \mathbb{D}$ such that $f(z_0) \in \mathbb{D}$, then the image $f(\mathbb{D})$ is completely contained in $\mathbb{D}$. To prove this claim, recall that the Open Mapping Theorem says that the image of an open set under a nonconstant analytic mapping is open.
\begin{Lemma}\label{jay5}
Let $f \in H^{\infty}(\mathbb{D})$ such that $||f||_{\infty} \leq 1$. If there is a point $z_0 \in \mathbb{D}$ such that $f(z_0) \in \mathbb{D}$, then $|f(z)| < 1$ for all $z \in \mathbb{D}$. 
\end{Lemma}
\begin{proof}
If $f$ is constant, then $f(z) = f(z_0)$ for all $z \in \mathbb{D}$, and the result follows. Otherwise, the Open Mapping Theorem and $||f||_{\infty}\leq 1$ imply that the image $f(\mathbb{D})$ is an open set contained in $\mathbb{D}$. Thus, $|f(z)| < 1$ for all $z \in \mathbb{D}$. 
\end{proof}

\begin{section}{Interpolation Theorems On The Subalgebra $H^{\infty}_K(\mathbb{D})$}\label{sect4}
We now present our interpolation theorems. The following theorem is a slight generalization of the interpolation theorem established in \cite{david}.

\begin{thm}\label{newtheorem1}
 Let $K = \{1,2,\dotsc,k\} \subset \mathbb{Z}_+$. Let $z_1,z_2,\dotsc,z_n$ be distinct points in $\mathbb{D}$ and $w_1,w_2,\dotsc,w_n$ be points in $\mathbb{D}$. There exists an analytic function $f \in H^{\infty}_K(\mathbb{D})$ with $f(z_i)=w_i$ for $i = 1, 2, \dotsc,n$ and $||f||_{\infty} \leq 1$ if and only if there exists a $\lambda \in \mathbb{D}$ such that
$$\left[\frac{z_i^{k+1} \, \, \overline{z^{k+1}_j} - \varphi_{\lambda}(w_i) \, \, \overline{\varphi_{\lambda}(w_j)}}{1-z_i\overline{z_j}}\right]\; \text{ is \; positive \; semidefinite.}$$
\end{thm}
\begin{proof}
Suppose $z_i \neq 0$ for each $i = 1,2,\dotsc, n$. Assume there exists an $f \in H^{\infty}_K(\mathbb{D})$ such that $f(z_i) = w_i$ and $||f||_{\infty} \leq 1$. Set $\lambda = f(0)$, which is in $\mathbb{D}$ by Lemma \ref{jay5}. Since $\phi_{\lambda} \in H^{\infty}(\mathbb{D})$ and $f \in H^{\infty}_K(\mathbb{D})$ with $||f||_{\infty}\leq 1$ we know from the Lemma \ref{jay2} that the composition $\phi_{\lambda} \circ f \in H^{\infty}_K(\mathbb{D})$. Further, $\phi\circ f$ at zero yields:
$$\phi_{\lambda}(f(0)) = \phi_{\lambda}(\lambda) = 0,$$
so $\phi_{\lambda}(f(z)) = z^{k+1}h(z)$ where $h \in H^{\infty}(\mathbb{D})$ and $||h||_{\infty} \leq 1$. Applying the Nevanlinna-Pick Interpolation Theorem on $\{z_1,z_2,\dotsc,z_n\}$ and $h$, the matrix 
$$\left[\frac{1-h(z_i)\overline{h(z_j)}}{1-z_i\overline{z_j}}\right]\; \; \text{is \; positive \; semidefinite.}$$
Given that $\phi_{\lambda}(f(z_i)) = \phi_{\lambda}(w_i) = z^{k+1}_ih(z_i)$ the matrix
$$\left[\frac{z_i^{k+1} \, \, \overline{z^{k+1}_j} - \phi_{\lambda}(w_i) \, \, \overline{\phi_{\lambda}(w_j)}}{1-z_i\overline{z_j}}\right]$$
equals
$$
\begin{bmatrix}
    z_1^{k+1} & 0 & ... & 0 \\
    0 & z_2^{k+1} & ... & 0\\
   \vdots & \vdots & \ddots & \vdots\\
   0 & 0 & \hdots & z_n^{k+1}
  \end{bmatrix}
\left [  \frac{1- h (z_i) \overline{ h (z_j)}  } { 1 - z_i \overline{z_j} } \right]  
\begin{bmatrix}
    \,\,\overline{z_1^{k+1}} & 0 & ... & 0 \\
    0 &  \overline{z_2^{k+1}} & ... & 0\\
   \vdots & \vdots & \ddots & \vdots\\
   0 & 0 & \hdots &  \overline{z_n^{k+1}}
  \end{bmatrix},
$$
and thus
$$\left[\frac{z^{k+1}_i \overline{z^{k+1}_j} - \varphi_{\lambda}(w_i) \, \overline{\varphi_{\lambda}(w_j)}}{1-z_i \overline{z_j}}\right] \; \;  \text{ is \; positive \; semidefinite.}$$

Now suppose $z_1 = 0$. Then $\lambda = f(0) = f(z_1) = w_1$, so the first row and column of the above matrix are zero; and the same reasoning applies to the remaining entries of the matrix.

Conversely, it suffices to assume that $z_i\neq 0$ for all $i = 1,2,\dotsc,n$. Suppose there exists a $\lambda \in \mathbb{D}$ such that 
$$\left[\frac{z^{k+1}_i \overline{z^{k+1}_j} - \varphi_{\lambda}(w_i) \, \overline{\varphi_{\lambda}(w_j)}}{1-z_i \overline{z_j}}\right] \; \; \text{is \; positive \; semidefinite.}$$
Since the above matrix equals
$$
\begin{bmatrix}
    z_1^{k+1} & 0 & ... & 0 \\
    0 & z_2^{k+1} & ... & 0\\
   \vdots & \vdots & \ddots & \vdots\\
   0 & 0 & \hdots & z_n^{k+1}
  \end{bmatrix}
\left [  \frac{1- z^{-(k+1)}_i \, \phi_{\lambda}(w_i) \, \overline{ z^{-(k+1)}_j \, \phi_{\lambda}(w_j)}  } { 1 - z_i \overline{z_j} } \right]  
\begin{bmatrix}
    \,\,\overline{z_1^{k+1}} & 0 & ... & 0 \\
    0 &  \overline{z_2^{k+1}} & ... & 0\\
   \vdots & \vdots & \ddots & \vdots\\
   0 & 0 & \hdots &  \overline{z_n^{k+1}}
  \end{bmatrix}
$$
we deduce that
$$\left [  \frac{1- z^{-(k+1)}_i \, \phi_{\lambda}(w_i) \, \overline{ z^{-(k+1)}_j \, \phi_{\lambda}(w_j)}  } { 1 - z_i \overline{z_j} } \right] \; \; \text{is \; positive \; semidefinite.}$$
By the Nevanlinna-Pick Interpolation Theorem, there is an $h \in H^{\infty} (\mathbb{D})$ such that $h(z_i) = z^{-(k+1)}_i\,\phi_{\lambda}(w_i)$, for each $i = 1,2,\dotsc,n$ and $||h||_{\infty}\leq 1$.

Now define $f(z) = \phi_{-\lambda}(z^{k+1}h(z))$. Clearly, $f \in H^{\infty}(\mathbb{D})$, $||f||_{\infty}\leq 1$, and for each $i = 1,2, \dotsc, n$, we have
$$f(z_i) = \phi_{-\lambda}(z^{k+1}_i \,  h(z_i)) = \phi_{-\lambda}(z^{k+1}_i \, z^{-(k+1)}_i\,\phi_{\lambda}(w_i)) = w_i.$$ 

To show that $f \in H^{\infty}_K(\mathbb{D})$, we first demonstrate that $z^{k+1}h(z) \in H^{\infty}_K(\mathbb{D})$. Clearly, $z^{k+1}h(z) \in H^{\infty}(\mathbb{D})$. Let $\ell \in K$ and consider the $\ell^{\text{th}}$ derivative of $z^{k+1}h(z)$:
$$\Big(z^{k+1}h(z)\Big)^{(\,\ell\,)} = \sum_{t=0}^{\ell}{\ell \choose t} \Big(z^{k+1}\Big)^{(\,t\,)} \Big(h(z)\Big)^{(\,\ell - t\,)}.$$
For each $t = 0, 1, \dotsc, k$, we have
$$\Big(z^{k+1}\Big)^{(\,t\,)} = \frac{(k+1)!}{(k-1+t)!} \cdot z^{k+1-t}.$$
Because $k + 1 - t > 0$, it follows that $\Big(z^{k+1}h(z)\Big)^{(\, \ell \, )} = 0$ at $z = 0$, for every $\ell \in K$. Thus, $z^{k+1}h(z) \in H^{\infty}_K(\mathbb{D})$. Finally, since $\phi_{-\lambda} \in H^{\infty}(\mathbb{D})$, we deduce from Lemma \ref{jay2} that $f(z) = \phi_{-\lambda}(z^{k+1}h(z)) \in H^{\infty}_K(\mathbb{D}),$ so $f$ is our desired interpolation function.
\end{proof}
\begin{remark}
If $K = \{1\}$, then $k+1 = 2$ gives us the interpolation result for $H^{\infty}_{\{1\}}(\mathbb{D})$. 
\end{remark}
 Given a nonempty set $K \subset \mathbb{Z}_+$ the following theorem and corollary give a sufficient condition for an interpolation function to exist in the algebra $H^{\infty}_K(\mathbb{D})$. Although the next theorem applies to sets $K = \{1,2,\dotsc,k\}$, it also applies to those sets $K$ with at least one pair of nonconsecutive integers for which $H^{\infty}_K(\mathbb{D})$ is an algebra, such as $H^{\infty}_{\{1,3\}}(\mathbb{D})$.
\begin{thm}\label{newtheorem2}
Let $K = \{1,k_1,k_2,\dotsc,k_p\}$ where $1 < k_1 < k_2 < \cdots < k_p$ such that $H^{\infty}_K(\mathbb{D})$ is an algebra. Let $z_1,z_2,\dotsc,z_n$ be distinct points in $\mathbb{D}$ and let $w_1, w_2, \dotsc, w_n$ be points in $\mathbb{D}$. If there exists a $\lambda \in \mathbb{D}$ such that the matrix
$$\left[\frac{z_i^{k_p+1} \, \, \overline{z^{k_p+1}_j} - \varphi_{\lambda}(w_i) \, \, \overline{\varphi_{\lambda}(w_j)}}{1-z_i\overline{z_j}}\right]\; \text{ is \; positive \; semidefinite,}$$
then there exists a function $f \in H^{\infty}_K(\mathbb{D})$ such that $f(z_i) = w_i$ and $||f||_{\infty} \leq 1$. 
\end{thm}
\begin{proof}
It is enough to assume that $z_i \neq 0$ for all $i = 1,2,\dotsc, n$. Given that the above matrix is positive semidefinite it follows that the matrix 
$$\left [  \frac{1- z^{-(k_p+1)}_i \, \phi_{\lambda}(w_i) \, \overline{ z^{-(k_p+1)}_j \, \phi_{\lambda}(w_j)}  } { 1 - z_i \overline{z_j} } \right]$$
is also positive semidefinite. By the Nevanlinna-Pick Interpolation Theorem, there exists a function $h \in H^{\infty}(\mathbb{D})$ such that $h(z_i) = z^{-(k_p + 1)}_i\phi_{\lambda}(w_i)$. Define the function $f(z) = \phi_{-\lambda}(z^{k_p + 1}h(z))$. A similar argument as in the proof of Theorem \ref{newtheorem1} shows that $f \in H^{\infty}_K(\mathbb{D}), \, f(z_i) = w_i$ and $||f||_{\infty}\leq 1$. 

\end{proof}
Theorem \ref{newtheorem2} and Lemma \ref{Kinfinite} allow us to deduce the following corollary.
\begin{cor}\label{newcorollary1}
Let $K$ be an infinite proper subset of $\mathbb{Z}_+$ such that $H^{\infty}_K(\mathbb{D})$ is an algebra and
\begin{equation*}\label{ugly} \mathbb{Z}_+ \setminus K =
\lbrace n_1d, n_2d, \dots , n_pd, N_0d, (N_0 + j)d : j \in \mathbb{Z}_+ \rbrace
\end{equation*}
for some positive integer $d$, some finite set $\{n_i\}_{i=1}^p \subset \mathbb{Z}_+$ where \\ $n_1 < n_2 < \cdots < n_p$ and a positive integer $N_0 > n_p$. Moreover, let $K_1 \subset \mathbb{Z}_+$ such that
\begin{equation*}\label{doubleugh}
\mathbb{Z}_+\setminus K_1 = \{n_1,n_2,\dotsc,n_p, N_0, N_0 + j \, : \, j \in \mathbb{Z}_+\}.
\end{equation*}
Let $z_1, z_2, \dotsc,z_n$ be distinct points in $\mathbb{D}$ and $w_1, w_2, \dotsc, w_n$ be points in $\mathbb{D}$. If there exists a $\lambda \in \mathbb{D}$ such that the matrix 
$$\left[\frac{z^{(n_1+1)d}_i \,\, \overline{z_j^{(n_1+1)d}} - \varphi_{\lambda}(w_i) \overline{\varphi_{\lambda}(w_j)}}{1-z_i^{ d}\,\overline{z_j^{ d}}}\right] \; \text{ is \; positive \; semidefinite}$$
where $n_1+1$ is the smallest positive integer not in $K_1$,
then there exists an $f \in H^{\infty}_K(\mathbb{D})$ with $f(z_i) = w_i$ for all $i = 1,2,\dotsc,n$ and $||f||_{\infty} \leq 1$. 
\end{cor}
\begin{proof}
Suppose there is a $\lambda \in \mathbb{D}$ such that the above matrix is positive semidefinite. Set $V_i = z^d_i$ and $\overline{V_j}=\overline{z^d_j}$, so that the matrix
$$\left[\frac{V^{n_1+1}_i \, \overline{V^{n_1+1}_j} - \varphi_{\lambda}(w_i) \overline{\varphi_{\lambda}(w_j)}}{1-V_i \, \overline{V_j}}\right] \; \; \text{is \; positive \; semidefinite.}$$
By Theorem \ref{newtheorem2}, there exists an $F \in H^{\infty}_{K_1}(\mathbb{D})$ such that $F(V_i) = w_i$ and $||F||_{\infty} \leq 1$. Further, we have
$$F(V) = \alpha_0 + \alpha_1 \, V^{n_1} + \alpha_2 \, V^{n_2} + \cdots + \alpha_p\,V^{n_p} + \alpha_{p+1}\,V^{N_0} + \sum_{j=1}^{\infty} \alpha_{p+j+1}\, V^{N_0+j}.$$
Defining $f(z) = F(z^d)$, we see that 
$$f(z) = \alpha_0 + \alpha_{1}\,z^{n_1d} + \alpha_{2}\,z^{n_2d} + 
\dots + \alpha_{p}\,z^{n_pd}  + \alpha_{p+1}\,z^{N_0d} + \sum_{j=1}^{\infty}\alpha_{p+j+1}\,z^{(N_0+j)d}.$$
Clearly, $f \in H^{\infty}_K(\mathbb{D})$, $||f||_{\infty} \leq 1$ and $f(z_i) = F(z^d_i) = w_i$, for all $i = 1,2,\dotsc,n$. We conclude that $f$ is our desired interpolation function.
\end{proof}

Unfortunately, it is not clear if the converse of the above theorem is true, nor is it apparent that a positive semidefinite matrix exists such that Theorem \ref{newtheorem2} becomes a biconditional statement. However, we provide a necessary condition for an interpolation function to exist in the following theorem.
\begin{thm}\label{newtheorem3}
 Let $K = \{1, k_1, k_2,\dotsc, k_p\}$ with $1 < k_1 < k_2 < \dotsc < k_p$ denote a set of positive integers such that $H^{\infty}_K(\mathbb{D})$ is an algebra. Let $z_1,z_2,\dotsc,z_n$ be distinct points in $\mathbb{D}$ and $w_1,w_2,\dotsc,w_n$ be points in $\mathbb{D}$. 

\noindent If there is an $f \in H^{\infty}_K(\mathbb{D})$ such that $f(z_i) = w_i$ for $i = 1,2,\dotsc,n$ and $||f||_{\infty}\leq 1$, then there exists a $\lambda \in \mathbb{D}$ such that the matrix is
$$\left[\frac{z_i^{m} \, \, \overline{z^{m}_j} - \varphi_{\lambda}(w_i) \, \, \overline{\varphi_{\lambda}(w_j)}}{1-z_i\overline{z_j}}\right]\; \text{ is \; positive \; semidefinite},$$
where $m$ is the smallest positive integer not in $K$.
\end{thm}
\begin{proof}
Without loss of generality, assume that $z_i \neq 0$ for all $i = 1,2,\dotsc,n$. Assume that there exists $f \in H^{\infty}_K(\mathbb{D})$ such that $||f||_{\infty} \leq 1$ and $f(z_i) = w_i$. Set $\lambda = f(0) \in \mathbb{D}$. Define $g(z) = \varphi_{\lambda}(f(z))$. Since $\varphi_{\lambda} \in H^{\infty}(\mathbb{D})$ and $f \in H^{\infty}_K(\mathbb{D})$, it follows from Lemma \ref{jay2} that $g \in H^{\infty}_{K}(\mathbb{D})$. Moreover, $g(z_i) = \varphi_{\lambda}(w_i),$ for $i = 1, 2\dotsc, n$, and $g(0) = \varphi_{\lambda}(\lambda) = 0$.

Consequently, we know from Lemma \ref{jay4} that a smallest positive integer $m\notin K$ exists such that $g(z) = z^m h(z)$ where $||h||_{\infty} \leq 1$. Applying the Nevanlinna-Pick Theorem on $h \in H^{\infty}(\mathbb{D})$ and the set $\{z_1,z_2,\dotsc,z_n\}$, the matrix 
$$\left[\frac{1-h(z_i)\overline{h(z_j)}}{1-z_i\overline{z_j}}\right]\; \; \text{is \; positive \; semidefinite.}$$ Following a similar idea in the proof of Theorem \ref{newtheorem1}, we have that the matrix
$$\left[\frac{z^{m}_i \overline{z^{m}_j} - g(z_i) \, \overline{g(z_j)}}{1-z_i \overline{z_j}}\right] = \left[\frac{z^{m}_i \overline{z^{m}_j} - \varphi_{\lambda}(w_i) \, \overline{\varphi_{\lambda}(w_j)}}{1-z_i \overline{z_j}}\right] \; \;  \text{ is \; positive \; semidefinite,}$$
as desired.
\end{proof}
\begin{cor}\label{newcorollary2}
Let $K$ be an infinite proper subset of $\mathbb{Z}_+$ such that $H^{\infty}_K(\mathbb{D})$ is an algebra and
\begin{equation*}\label{ugly} \mathbb{Z}_+ \setminus K =
\lbrace n_1d, n_2d, \dots , n_pd, N_0d, (N_0 + j)d : j \in \mathbb{Z}_+ \rbrace
\end{equation*}
for some positive integer $d$, some finite set $\{n_i\}_{i=1}^p \subset \mathbb{Z}_+$ where \\ $n_1 < n_2 < \cdots < n_p$ and a positive integer $N_0 > n_p$. Moreover, let $K_1 \subset \mathbb{Z}_+$ such that
\begin{equation*}\label{doubleugh}
\mathbb{Z}_+\setminus K_1 = \{n_1,n_2,\dotsc,n_p, N_0, N_0 + j \, : \, j \in \mathbb{Z}_+\}.
\end{equation*}
Let $z_1, z_2, \dotsc,z_n$ be distinct points in $\mathbb{D}$ and $w_1, w_2, \dotsc, w_n$ be points in $\mathbb{D}$. If there exists an $f \in H^{\infty}_K(\mathbb{D})$ with $f(z_i) = w_i$ for $i = 1,2,\dotsc,n$ and $||f||_{\infty} \leq 1$, then there is a $\lambda \in \mathbb{D}$ such that 
$$\left[\frac{z^{n_1d}_i \overline{z_j^{n_1d}} - \varphi_{\lambda}(w_i) \overline{\varphi_{\lambda}(w_j)}}{1-z_i^{ d}\,\overline{z_j^{ d}}}\right] \; \text{ is \; positive \; semidefinite,}$$
where $n_1$ is the smallest positive integer not in $K_1$.
\end{cor}

\begin{proof}
Let $f \in H^{\infty}_K(\mathbb{D})$ such that $f(z_i) = w_i$ and $||f||_{\infty} \leq 1$. 
By our assumption on $\mathbb{Z}_+\setminus K$, we see that $f$ has the following representation
$$f(z) = \alpha_0 + \alpha_{1}\,z^{n_1d} + \alpha_{2}\,z^{n_2d} + 
\dots + \alpha_{p}\,z^{n_pd}  + \alpha_{p+1}\,z^{N_0d} + \sum_{j=1}^{\infty}\alpha_{p+j+1}\,z^{(N_0+j)d},$$
where $\alpha_i \in \mathbb{C}$. Setting $V = z^d$, we have a function $F$ defined by
$$F(V) = \alpha_0 + \alpha_1 \, V^{n_1} + \alpha_2 \, V^{n_2} + \cdots +\alpha_p\,V^{n_p} + \alpha_{p+1}\,V^{N_0} +  \sum_{j=1}^{\infty} \alpha_{p+j+1}\, V^{N_0+j}.$$
Consequently, $F$ is contained in the algebra $H^{\infty}_{K_1}(\mathbb{D})$. 
Also, $||F||_{\infty} \leq 1$ since $||f||_{\infty} \leq 1$, and $F(V_i) = F(z^d_i) = f(z_i) = w_i$. By Theorem \ref{newtheorem3}, there exists a $\lambda \in \mathbb{D}$ such that 
$$\left[\frac{V^{n_1}_i \, \overline{V^{n_1}_j} - \varphi_{\lambda}(V_i) \overline{\varphi_{\lambda}(V_j)}}{1-V_i \, \overline{V_j}}\right] \; \; \text{is \; positive \; semidefinite}$$
where $n_1$ is the smallest positive integer not in $K_1$. Since $V_i = z^d_i$ and $\overline{V_j} = \overline{z^d_j}$, we conclude that
$$\left[\frac{z^{n_1d}_i \overline{z_j^{n_1d}} - \varphi_{\lambda}(w_i) \overline{\varphi_{\lambda}(w_j))}}{1-z_i^{d}\,\overline{z_j^{d}}}\right] \; \text{ is \; positive \; semidefinite}.$$
\end{proof}
\end{section}

\begin{section}{Conclusion and Open Problems}
In this paper, we have established several interpolation results for $H^{\infty}_K(\mathbb{D})$. In particular, we have generalized the interpolation result in \cite{david}, and have provided a couple theorems that serve as necessary and sufficient conditions for an interpolation function to exist. However, questions  have naturally arisen during this investigation.

First, we gave two different theorems on the existence of an interpolation function in $H^{\infty}_K(\mathbb{D})$ when $K$ is an arbitrary set of positive integers such that $H^{\infty}_K(\mathbb{D})$ is an algebra. However, it is not clear if there is a necessary and sufficient condition for an interpolant to exist in algebras such as $H^{\infty}_{\{1,3\}}(\mathbb{D})$. Thus, we ask the following question.
\begin{question}
\normalfont Is there a generalization of Theorem \ref{newtheorem1} for algebras $H^{\infty}_K(\mathbb{D})$ where $K$ is a collection of nonconsecutive integers?
\end{question}
 Furthermore, the interpolation function in the classical Nevanlinna-Pick interpolation result is unique if and only if the associated Pick matrix is singular. This fact inspires the following question.
\begin{question}
\normalfont Are there necessary and sufficient conditions for Theorem \ref{newtheorem1} to guarantee that the the interpolation function in $H^{\infty}_K(\mathbb{D})$ is unique?
\end{question}
Lemma \ref{jay2} shows that if $g \in H^{\infty}(\mathbb{D})$ and $f \in H^{\infty}_K(\mathbb{D})$ where $f(\mathbb{D})\subset\mathbb{D}$, then the composition $g \circ f \in H^{\infty}_K(\mathbb{D})$. We thus pose the following question.
\begin{question}
\normalfont Suppose $g \in H^{\infty}_K(\mathbb{D})$ and $f \in H^{\infty}(\mathbb{D})$ with $f(\mathbb{D}) \subset \mathbb{D}$. Under what conditions is the composition $g\circ f\in H^{\infty}_K(\mathbb{D})$?
\end{question}
\end{section}
\noindent\textbf{Acknowledgements:} The authors would like to thank A. Incognito and Tavan T. Trent  for important feedback on the first version of this manuscript.

\end{document}